\documentclass[11pt,reqno]{amsart}
\usepackage{fullpage}
\usepackage{mathrsfs,amssymb,graphicx,verbatim,amsmath,amsfonts}
\usepackage{paralist}
\usepackage[breaklinks,pdfstartview=FitH]{hyperref}
\usepackage{upgreek}
\usepackage{tikz}
\usepackage[mathscr]{euscript}

\hypersetup{ocgcolorlinks=true,allcolors=testc}
\hypersetup{
     colorlinks   = true,
     citecolor    = black
}
\hypersetup{linkcolor=black}
\hypersetup{urlcolor=black}

\usepackage{dutchcal}

\renewcommand{\leq}{\leqslant}

\renewcommand{\gamma}{\upgamma}
\allowdisplaybreaks

\usepackage{esint}
\usepackage[autostyle]{csquotes}

\renewcommand{\pi}{\uppi}

\newcommand{\e}{\varepsilon}
\newcommand{\R}{\mathbb R}

\newtheorem{theorem}{Theorem}
\newtheorem{lemma}[theorem]{Lemma}
\newtheorem{proposition}[theorem]{Proposition}

\newtheorem{corollary}[theorem]{Corollary}

\theoremstyle{remark}
\newtheorem{remark}[theorem]{Remark}

\renewcommand{\tau}{\uptau}

\renewcommand{\xi}{\upxi}
\renewcommand{\rho}{\uprho}

\newcommand{\N}{\mathbb N}

\newcommand{\eqdef}{\stackrel{\mathrm{def}}{=}}

\renewcommand{\theta}{\uptheta}
\renewcommand{\lambda}{\uplambda}

\renewcommand{\emptyset}{\varnothing}

\renewcommand{\gamma}{\upgamma}
\renewcommand{\beta}{\upbeta}
\renewcommand{\alpha}{\upalpha}
\renewcommand{\kappa}{\upkappa}
\renewcommand{\psi}{\uppsi}
\renewcommand{\rho}{\uprho}
\renewcommand{\delta}{\updelta}
\renewcommand{\pi}{\uppi}
\renewcommand{\omega}{\upomega}
\renewcommand{\sigma}{\upsigma}

\renewcommand{\eta}{\upeta}

\renewcommand{\kappa}{\upkappa}
\renewcommand{\mu}{\upmu}
\renewcommand{\nu}{\upnu}
\renewcommand{\pi}{\uppi}
\renewcommand{\zeta}{\upzeta}

\newcommand{\mb}{\mathbb}
\newcommand*\diff{\mathop{}\!\mathrm{d}}

\newcommand{\ms}{\mathscr}
\newcommand{\msf}{\mathsf}
\newcommand{\mf}{\mathfrak}

\begin{document}

\title{On Pisier's inequality for UMD targets}

\author{Alexandros Eskenazis}
\address{Institut de Math\'ematiques de Jussieu
\\ Sorbonne Universit\'e\\ 4, Place Jussieu\\ 75252 Paris Cedex 05\\ France}
\email{alexandros.eskenazis@imj-prg.fr}
%\address{Institut de Math\'ematiques de Jussieu, Sorbonne Universit\'e, Paris, 75252, France}
%\email{alexandros.eskenazis@imj-prg.fr}

\thanks{The author was supported by a postdoctoral fellowship of the Fondation Sciences Math\'ematiques de Paris.}

\maketitle

\vspace{-0.25in}

\begin{abstract} 
We prove an extension of Pisier's inequality (1986) with a dimension independent constant for vector valued functions whose target spaces satisfy a relaxation of the UMD property.

%\medskip

%\noindent {\scshape Resum\'e.} One montre une extension de l'inégalité de Pisier (1986) avec une constante indépendante de la dimension pour les fonctions prenant des valeurs dans les espaces de Banach satisfaisant une relaxation de la propriété UMD.
\end{abstract}

\medskip

{\footnotesize
\noindent {\em 2010 Mathematics Subject Classification.} Primary: 46B07; Secondary: 46B85, 42C10, 60G46.

\noindent {\em Key words.} Pisier's inequality, Banach space valued martingales, UMD Banach spaces.}

\section{Introduction}

Let $(X,\|\cdot\|_X)$ be a Banach space. For $p\in[1,\infty)$, the vector valued $L_p$ norm of a function $f:\Omega\to X$ defined on a measure space $(\Omega,\ms{F},\mu)$ is given by $\|f\|_{L_p(\Omega,\mu;X)}^p =  \int_\Omega \|f(\omega)\|_X^p\diff\mu(\omega)$. When $\Omega$ is a finite set and $\mu$ is the normalized counting measure, we will simply write $\|f\|_{L_p(\Omega;X)}$. 

Let $\ms{C}_n=\{-1,1\}^n$ be the discrete hypercube. For $i\in\{1,\ldots,n\}$, the $i$-th partial derivative of a function $f:\ms{C}_n\to X$ is defined by
\begin{equation}
\forall \ \e\in\ms{C}_n, \ \ \ \partial_if(\e) \eqdef \frac{f(\e)-f(\e_1,\ldots,\e_{i-1},-\e_i,\e_{i+1},\ldots,\e_n)}{2}.
\end{equation}
 In \cite{Pis86}, Pisier showed that for every $n\in\N$ and $p\in[1,\infty)$, \mbox{every $f:\ms{C}_n\to X$ satisfies}
\begin{equation} \label{eq:pisier}
\Big\|f - \frac{1}{2^n} \sum_{\delta\in\ms{C}_n} f(\delta) \Big\|_{L_p(\ms{C}_n;X)} \leq \mf{P}_p^n(X) \Big( \frac{1}{2^n} \sum_{\delta\in\ms{C}_n} \Big\|\sum_{i=1}^n \delta_i \partial_i f \Big\|_{L_p(\ms{C}_n;X)}^p\Big)^{1/p},
\end{equation}
with $\mf{P}_p^n(X) = 2e\log n$. Showing that $\mf{P}_p^n(X)$ is bounded by a constant depending only on $p$ and the geometry of the given Banach space $X$, is of fundamental importance in the theory of nonlinear type (see \cite{Pis86, NS02}). The first positive and negative results in this direction were obtained by Talagrand in \cite{Tal93}, who showed that $\mf{P}_p^n(\R) = \Theta(1)$ and $\mf{P}_p^n(\ell_\infty) = \Theta(\log n)$ for every $p\in[1,\infty)$.

Talagrand's dimension independent scalar valued inequality \eqref{eq:pisier} was greatly generalized in the range $p\in(1,\infty)$ by Naor and Schechtman \cite{NS02}. Recall that a Banach space $(X,\|\cdot\|_X)$ is called a UMD space if for every $p\in(1,\infty)$, there exists a constant $\beta_p\in(0,\infty)$ such that for every $n\in\N$, every probability space $(\Omega, \ms{F},\mu)$ and every filtration $\{\ms{F}_i\}_{i=0}^n$ of sub-$\sigma$-algebras of $\ms{F}$, every martingale $\{\ms{M}_i:\Omega \to X\}_{i=0}^n$ adapted to $\{\ms{F}_i\}_{i=0}^n$ satisfies
\begin{equation} \label{eq:umd}
\max_{\delta=(\delta_1,\ldots,\delta_n)\in\ms{C}_n}\Big\| \sum_{i=1}^n \delta_i (\ms{M}_i-\ms{M}_{i-1}) \Big\|_{L_p(\Omega,\mu;X)} \leq \beta_p \|\ms{M}_n-\ms{M}_0\|_{L_p(\Omega,\mu;X)}.
\end{equation}
The least constant $\beta_p\in(0,\infty)$ for which \eqref{eq:umd} holds is called the UMD$_p$ constant of $X$ and is denoted by $\beta_p(X)$. In \cite{NS02}, Naor and Schechtman proved that for every UMD Banach space $X$ and $p\in(1,\infty)$,
\begin{equation} \label{eq:ns}
\sup_{n\in\N}\mf{P}_p^n(X) \leq \beta_p(X).
\end{equation}
Their result was later strengthened by  Hyt\"onen  and Naor \cite{HN13} in terms of the random martingale transform inequalities of Garling, see \cite{Gar90}. Recall that a Banach space $(X,\|\cdot\|_X)$ is a UMD$^+$ space if for every $p\in(1,\infty)$ there exists a constant $\beta_p^+\in(0,\infty)$ such that for every martingale $\{\ms{M}_i:\Omega \to X\}_{i=0}^n$ as before, we have
\begin{equation} \label{eq:umd+}
\Big(\frac{1}{2^n} \sum_{\delta\in\ms{C}_n} \Big\| \sum_{i=1}^n \delta_i (\ms{M}_i-\ms{M}_{i-1}) \Big\|^p_{L_p(\Omega,\mu;X)}\Big)^{1/p} \leq \beta_p^+ \|\ms{M}_n-\ms{M}_0\|_{L_p(\Omega,\mu;X)}.
\end{equation}
Similarly, $X$ is a UMD$^-$ Banach space if for every $p\in(1,\infty)$ there exists a constant $\beta_p^-\in(0,\infty)$ such that for every martingale $\{\ms{M}_i:\Omega \to X\}_{i=0}^n$ as before, we have
\begin{equation} \label{eq:umd-}
\|\ms{M}_n-\ms{M}_0\|_{L_p(\Omega,\mu;X)} \leq \beta_p^- \Big(\frac{1}{2^n} \sum_{\delta\in\ms{C}_n} \Big\| \sum_{i=1}^n \delta_i (\ms{M}_i-\ms{M}_{i-1}) \Big\|^p_{L_p(\Omega,\mu;X)}\Big)^{1/p}.
\end{equation}
The least positive constants $\beta_p^+, \beta_p^-$ for which \eqref{eq:umd+} and \eqref{eq:umd-} hold are respectively called the UMD$_p^+$ and UMD$_p^-$ constants of $X$ and denoted by $\beta_p^+(X)$ and $\beta_p^-(X)$. In \cite{HN13}, Hyt\"onen and Naor showed that for every Banach space $X$ whose dual $X^\ast$ is a UMD$^+$ space and $p\in(1,\infty)$,
\begin{equation} 
\sup_{n\in\N} \mf{P}_p^n(X) \leq \beta_{p/(p-1)}^+(X^\ast).
\end{equation}
In fact, in \cite[Theorem~1.4]{HN13}, the authors proved a generalization (see \eqref{eq:hn}) of inequality \eqref{eq:pisier} for a family of $n$ functions $\{f_i:\ms{C}_n\to X\}_{i=1}^n$ under the assumption that the dual of $X$ is UMD$^+$.

The main result of the present note is a different inequality of this nature with respect to a Fourier analytic parameter of $X$. For a Banach space $(X,\|\cdot\|_X)$ and $p\in(1,\infty)$, let $\mf{s}_p(X)\in(0,\infty]$ be the least constant $\mf{s} \in(0,\infty]$ such that the following holds. For every probability space $(\Omega, \ms{F},\mu)$, $n\in\N$ and filtration $\{\ms{F}_i\}_{i=1}^n$ of sub-$\sigma$-algebras of $\ms{F}$ with corresponding vector valued conditional expectations $\{\ms{E}_i\}_{i=1}^n$, every sequence of functions $\{f_i:\Omega\to X\}_{i=1}^n$ satisfies
\begin{equation} \label{eq:bourgain}
\Big(\frac{1}{2^n} \sum_{\delta\in\ms{C}_n} \Big\| \sum_{i=1}^n \delta_i\ms{E}_i f_i\Big\|_{L_p(\Omega,\mu;X)}^p\Big)^{1/p} \leq \mf{s}\Big( \frac{1}{2^n} \sum_{\delta\in\ms{C}_n} \Big\| \sum_{i=1}^n \delta_if_i\Big\|_{L_p(\Omega,\mu;X)}^p\Big)^{1/p}.
\end{equation}
The square function inequality \eqref{eq:bourgain} originates in Stein's classical work \cite{Ste70}, where he showed that $\mf{s}_p(\R) = \Theta(1)$ for every $p\in(1,\infty)$. In the vector valued setting which is of interest here, it has been proven by Bourgain in \cite{Bou86} that for every UMD$^+$ Banach space and $p\in(1,\infty)$,
\begin{equation} \label{eq:bou}
\mf{s}_p(X) \leq\beta_p^+(X).
\end{equation}
For a function $f:\ms{C}_n\to X$ and $i\in\{0,1,\ldots,n\}$ denote by
\begin{equation} \label{eq:defEi}
\forall \ \e\in\ms{C}_n, \ \ \ \ms{E}_i f (\e) \eqdef \frac{1}{2^{n-i}} \sum_{\delta_{i+1},\ldots,\delta_n\in\{-1,1\}} f(\e_1,\ldots,\e_i,\delta_{i+1},\ldots,\delta_n),
\end{equation}
so that $\ms{E}_n f = f$ and $\ms{E}_0 f = \frac{1}{2^n} \sum_{\delta\in\ms{C}_n} f(\delta)$. The main result of this note is the following theorem.

\begin{theorem} \label{thm}
Fix $p\in(1,\infty)$ and let $(X,\|\cdot\|_X)$ be a Banach space with $\mf{s}_p(X)<\infty$. If, additionally, $X$ is a UMD$^-$ space, then for every $n\in\N$ and functions $f_1,\ldots,f_n:\ms{C}_n\to X$, we have
\begin{equation} \label{eq}
\begin{split}
\Big\| \sum_{i=1}^n (\ms{E}_i f_i - \ms{E}_{i-1}f_i)\Big\|_{L_p(\ms{C}_n;X)} \leq \mf{s}_p(X) \beta_p^-(X) \Big(\frac{1}{2^n} \sum_{\delta\in\ms{C}_n} \Big\| \sum_{i=1}^n \delta_i \partial_i f_i\Big\|_{L_p(\ms{C}_n;X)}^p\Big)^{1/p}.
\end{split}
\end{equation}
Choosing $f_1=\cdots=f_n=f$, we deduce that the constants in Pisier's inequality \eqref{eq:pisier} satisfy
\begin{equation} \label{eq:our}
\sup_{n\in\N} \mf{P}_p^n(X) \leq \mf{s}_p(X) \beta_p^-(X).
\end{equation}
\end{theorem}

Combining \eqref{eq:our} with Bourgain's inequality \eqref{eq:bou}, we deduce that $\sup_{n\in\N} \mf{P}_p^n(X) \leq \beta_p^+(X)\beta_p^-(X)$, which is weaker than Naor and Schechtman's bound \eqref{eq:ns}. Nevertheless, it appears to be unknown (see \cite[p.~197]{Pis16}) whether every Banach space $X$ with $\mf{s}_p(X)<\infty$ is necessarily a UMD$^+$ space. Therefore, it is conceivable that there exist Banach spaces $X$ for which inequality \eqref{eq:our} does not follow from the previously known results of \cite{NS02,HN13}. We will see in Proposition \ref{fact:hn1} below that if the dual $X^\ast$ of a Banach space $X$ is UMD$^+$, then $X$ satisfies the assumptions of Theorem \ref{thm}. Therefore, Theorem \ref{thm} also contains the aforementioned result of \cite{HN13}. 

Moreover, Theorem \ref{thm} implies an inequality similar to \cite[Theorem~1.4]{HN13} (see also Remark \ref{rem} below for comparison), under different assumptions. We will need some standard terminology from discrete Fourier analysis. Recall that every function \mbox{$f:\ms{C}_n\to X$ can be expanded in a Walsh series as}
\begin{equation} \label{eq:walsh}
f = \sum_{A\subseteq\{1,\ldots,n\}} \widehat{f}(A) w_A,
\end{equation}
where the Walsh function $w_A:\ms{C}_n\to\{-1,1\}$ is given by $w_A(\e) = \prod_{i\in A}\e_i$ for $\e=(\e_1,\ldots,\e_n)\in\ms{C}_n$, and $\widehat{f}(A)\in X$. Moreover, the fractional hypercube Laplacian of\mbox{ a function $f:\ms{C}_n\to X$ is given by}
\begin{equation}
\forall \ \alpha\in\R, \ \ \ \ \Delta^\alpha\Big( \sum_{A\subseteq\{1,\ldots,n\}} \widehat{f}(A) w_A \Big) \eqdef  \sum_{\substack{A\subseteq\{1,\ldots,n\} \\ A\neq\emptyset}} |A|^\alpha \widehat{f}(A) w_A.
\end{equation}

\begin{corollary} \label{cor}
Fix $p\in(1,\infty)$ and let $(X,\|\cdot\|_X)$ be a Banach space with $\mf{s}_p(X)<\infty$. If, additionally, $X$ is a UMD$^-$ space, then for every $n\in\N$ and functions $f_1,\ldots,f_n:\ms{C}_n\to X$, we have
\begin{equation} \label{eq:cor}
\begin{split}
\Big\| \sum_{i=1}^n \Delta^{-1}\partial_i f_i \Big\|_{L_p(\ms{C}_n;X)} \leq \mf{s}_p(X) \beta_p^-(X) \Big(\frac{1}{2^n} \sum_{\delta\in\ms{C}_n} \Big\| \sum_{i=1}^n \delta_i \partial_i f_i\Big\|_{L_p(\ms{C}_n;X)}^p\Big)^{1/p}.
\end{split}
\end{equation}
\end{corollary}

\smallskip

\noindent {\bf Acknowledgements.} I would like to thank Assaf Naor for helpful discussions.

\smallskip

\section{Proofs}

We first present the proof of Theorem \ref{thm}.

\medskip

\noindent {\it Proof of Theorem \ref{thm}.} For a function $h:\ms{C}_n\to X$ and $i\in\{1,\ldots,n\}$ consider the averaging operator
\begin{equation} \label{eq:crucialident}
\forall \ \e\in\ms{C}_n, \ \ \ \ \msf{E}_i h(\e) \eqdef \frac{h(\e)+h(\e_1,\ldots,\e_{i-1},-\e_i,\e_{i+1},\ldots,\e_n)}{2} = (\msf{id}-\partial_i)h(\e),
\end{equation}
where $\msf{id}$ is the identity operator. Then, for every $i\in\{0,1,\ldots,n\}$ we have the identities
\begin{equation}
\ms{E}_i h= \msf{E}_{i+1} \circ\cdots\circ \msf{E}_n h = \mb{E}[h|\ms{F}_i],
\end{equation}
where $\ms{F}_i = \sigma(\e_1,\ldots,\e_i)$. Since for every $i\in\{1,\ldots,n\}$,
\begin{equation}
\mb{E}\big[ \ms{E}_i f_i - \ms{E}_{i-1} f_i \big| \ms{F}_{i-1}\big] = 0,
\end{equation}
the sequence $\{\ms{E}_i f_i - \ms{E}_{i-1} f_i \}_{i=1}^n$ is a martingale difference sequence and thus the UMD$^-$ condition and \eqref{eq:bourgain} imply that
\begin{equation} \label{long ineq}
\begin{split}
\Big\| \sum_{i=1}^n (\ms{E}_i f_i  - \ms{E}_{i-1}f_i)\Big\|_{L_p(\ms{C}_n;X)} 
& \stackrel{\eqref{eq:umd-}}{\leq} \beta_p^-(X) \Big( \frac{1}{2^n} \sum_{\delta\in\ms{C}_n} \Big\| \sum_{i=1}^n \delta_i (\ms{E}_i f_i - \ms{E}_{i-1}f_i)\Big\|_{L_p(\ms{C}_n;X)}^p\Big)^{1/p}
\\ & \stackrel{\eqref{eq:crucialident}}{=} \beta_p^-(X) \Big( \frac{1}{2^n} \sum_{\delta\in\ms{C}_n} \Big\| \sum_{i=1}^n \delta_i \ms{E}_i \partial_i f_i\Big\|_{L_p(\ms{C}_n;X)}^p\Big)^{1/p} 
\\ & \stackrel{\eqref{eq:bourgain}}{\leq} \mf{s}_p(X) \beta_p^-(X) \Big(\frac{1}{2^n} \sum_{\delta\in\ms{C}_n} \Big\| \sum_{i=1}^n \delta_i \partial_i f_i\Big\|_{L_p(\ms{C}_n;X)}^p\Big)^{1/p}
\end{split},
\end{equation}
which completes the proof.
\hfill$\Box$

\medskip

We will now derive Corollary \ref{cor} from Theorem \ref{thm}. The proof follows a symmetrization argument of \cite{HN13}.

\medskip

\noindent {\it Proof of Corollary \ref{cor}.} As noticed in \eqref{long ineq} above, \eqref{eq} can be equivalently written as
\begin{equation}
\Big\| \sum_{i=1}^n \ms{E}_i \partial_i f_i \Big\|_{L_p(\ms{C}_n;X)} \leq \mf{s}_p(X) \beta_p^-(X) \Big(\frac{1}{2^n} \sum_{\delta\in\ms{C}_n} \Big\| \sum_{i=1}^n \delta_i \partial_i f_i\Big\|_{L_p(\ms{C}_n;X)}^p\Big)^{1/p}.
\end{equation}
Fix a permutation $\pi\in S_n$ and consider the filtration $\{\ms{F}_i^\pi\}_{i=0}^n$ given by $\ms{F}_i^\pi = \sigma(\e_{\pi(1)},\ldots,\e_{\pi(i)})$ with corresponding conditional expectations $\{\ms{E}_i^\pi\}_{i=0}^n$. Repeating the argument of the proof of Theorem \ref{thm} for this filtration and the martingale difference sequence $\{\ms{E}_i^\pi f_{\pi(i)} - \ms{E}_{i-1}^\pi f_{\pi(i)}\}_{i=1}^n$, we see that for every $\pi \in S_n$,
\begin{equation} \label{avr ineq}
\begin{split}
\Big\| \sum_{i=1}^n \ms{E}_i^\pi \partial_{\pi(i)} f_{\pi(i)} \Big\|_{L_p(\ms{C}_n;X)} & \leq \mf{s}_p(X) \beta_p^-(X) \Big(\frac{1}{2^n} \sum_{\delta\in\ms{C}_n} \Big\| \sum_{i=1}^n \delta_i \partial_{\pi(i)} f_{\pi(i)}\Big\|_{L_p(\ms{C}_n;X)}^p\Big)^{1/p}
\\ & =\mf{s}_p(X) \beta_p^-(X) \Big(\frac{1}{2^n} \sum_{\delta\in\ms{C}_n} \Big\| \sum_{i=1}^n \delta_i \partial_i f_i\Big\|_{L_p(\ms{C}_n;X)}^p\Big)^{1/p},
\end{split}
\end{equation}
since $(\delta_1,\ldots,\delta_n)$ has the same distribution as $(\delta_{\pi(1)},\ldots,\delta_{\pi(n)})$.
An obvious adaptation of \eqref{eq:defEi} along with \eqref{eq:walsh} shows that for every $h:\ms{C}_n\to X$,
\begin{equation}
\ms{E}_i^\pi h = \sum_{A\subseteq\{\pi(1),\ldots,\pi(i)\}} \widehat{h}(A) w_A
\end{equation}
where $\widehat{h}(A)$ are the Walsh coefficients of $h$. Therefore, expanding each $f_{\pi(i)}$ as a Walsh series \eqref{eq:walsh} we have
\begin{equation}
\forall \ i\in\{1,\ldots,n\}, \ \ \ \ \ms{E}^\pi_i\partial_{\pi(i)} f_{\pi(i)} = \sum_{\substack{A\subseteq\{1\ldots,n\} \\ \max \pi^{-1}(A)=i}} \widehat{f_{\pi(i)}}(A) w_A
\end{equation}
and therefore
\begin{equation} \label{eq:av perm}
\sum_{i=1}^n \ms{E}^\pi_i\partial_{\pi(i)} f_{\pi(i)} = \sum_{A\subseteq\{1,\ldots,n\}} \widehat{f_{\pi(\max \pi^{-1}(A))}}(A) w_A.
\end{equation}
Averaging \eqref{eq:av perm} over all permutations $\pi\in S_n$ and using the fact that $\pi(\max\pi^{-1}(A))$ is uniformly distributed in $A$, we get
\begin{equation*}
\frac{1}{n!}\sum_{\pi\in S_n} \sum_{i=1}^n \ms{E}^\pi_i\partial_{\pi(i)} f_{\pi(i)} =\!\!\! \sum_{\substack{A\subseteq\{1,\ldots,n\}\\ A\neq\emptyset}} \frac{1}{|A|} \sum_{i\in A} \widehat{f_i}(A) w_A = \sum_{i=1}^n \sum_{\substack{A\subseteq\{1,\ldots,n\} \\ i\in A}} \frac{1}{|A|} \widehat{f_i}(A) w_A = \sum_{i=1}^n \Delta^{-1}\partial_i f_i.
\end{equation*}
Hence, by convexity we finally deduce that
\begin{equation}
\begin{split}
\Big\| \sum_{i=1}^n \Delta^{-1}\partial_i f_i\Big\|_{L_p(\ms{C}_n;X)} & \leq \frac{1}{n!}\sum_{\pi\in S_n} \Big\| \sum_{i=1}^n \ms{E}_i^\pi \partial_{\pi(i)} f_{\pi(i)} \Big\|_{L_p(\ms{C}_n;X)}
\\ & \stackrel{\eqref{avr ineq}}{\leq} \mf{s}_p(X) \beta_p^-(X) \Big(\frac{1}{2^n} \sum_{\delta\in\ms{C}_n} \Big\| \sum_{i=1}^n \delta_i \partial_i f_i\Big\|_{L_p(\ms{C}_n;X)}^p\Big)^{1/p},
\end{split}
\end{equation}
which completes the proof.
\hfill$\Box$

\begin{remark} \label{rem}
In \cite{HN13}, Hyt\"onen and Naor obtained a different extension of Pisier's inequality \eqref{eq:pisier} for Banach spaces whose dual is UMD$^+$. For a function $F:\ms{C}_n\times\ms{C}_n\to X$ and $i\in\{1,\ldots,n\}$, let $F_i:\ms{C}_n\to X$ be given by
\begin{equation}
\forall \ \e\in\ms{C}_n, \ \ \ \ F_i(\e) \eqdef \frac{1}{2^n}\sum_{\delta\in\ms{C}_n} \delta_i F(\e,\delta).
\end{equation}
In \cite[Theorem~1.4]{HN13}, it was shown that for every $p\in(1,\infty)$ and every function $F:\ms{C}_n\times\ms{C}_n\to X$,
\begin{equation} \label{eq:hn0}
\Big\| \sum_{i=1}^n \Delta^{-1}\partial_i F_i\Big\|_{L_p(\ms{C}_n;X)} \leq \beta_{p/(p-1)}^+(X^\ast) \|F\|_{L_p(\ms{C}_n\times\ms{C}_n;X)}.
\end{equation}
In fact, since every Banach space whose dual is UMD$^+$ is $K$-convex (see \cite{Pis16} and Section \ref{sec:3} below) the validity of inequality \eqref{eq:hn0} is equivalent to its validity for functions of the form $F(\e,\delta) = \sum_{i=1}^n\delta_i F_i(\e)$, where $F_1,\ldots,F_n:\ms{C}_n\to X$. In other words, \cite[Theorem~1.4]{HN13} is equivalent to the fact that if $X^\ast$ is UMD$^+$, then for every $F_1,\ldots,F_n:\ms{C}_n\to X$ and $p\in(1,\infty)$,
\begin{equation} \label{eq:hn}
\Big\| \sum_{i=1}^n \Delta^{-1}\partial_i F_i\Big\|_{L_p(\ms{C}_n;X)} \leq A_p(X)  \Big(\frac{1}{2^n} \sum_{\delta\in\ms{C}_n} \Big\| \sum_{i=1}^n \delta_i  F_i\Big\|_{L_p(\ms{C}_n;X)}^p\Big)^{1/p},
\end{equation}
up to the value of the constant $A_p(X)$. In particular, applying \eqref{eq:hn} to $F_i=\partial_i f_i$, one recovers Corollary \ref{cor}, so inequality \eqref{eq:hn} of \cite{HN13} is formally stronger than \eqref{eq:cor} in the class of spaces whose dual is UMD$^+$.
\end{remark}

\section{Concluding remarks} \label{sec:3}

In this section we will compare our result with existing theorems in the literature. Recall that a Banach $X$ space is $K$-convex if $X$ does not contain the family $\{\ell_1^n\}_{n=1}^\infty$ with uniformly bounded distortion. We will need the following lemma.
 
 \begin{lemma} \label{f2}
If a space $(X,\|\cdot\|_X)$ satisfies $\mf{s}_p(X)<\infty$ for some \mbox{$p\in(1,\infty)$, then $X$ is $K$-convex.}
 \end{lemma}
 
 \begin{proof}
It is well known since Stein's work \cite{Ste70} that inequality \eqref{eq:bourgain} does not hold for $p\in\{1,\infty\}$ even for scalar valued functions. In fact, an inspection of the argument in \cite[p.~105]{Ste70} shows that for every $n\in\N$ there exists $n$ functions $g_1,\ldots,g_n:\ms{C}_n\to\{0,1\}$ such that for every $q\in(2,\infty)$,
\begin{equation}
\Big\| \Big(\sum_{i=1}^n \big(\ms{E}_i g_i\big)^2\Big)^{1/2}\Big\|_{L_q(\ms{C}_n;\R)} \gtrsim \Big( \int_0^n y^{q/2} e^{-y}\diff y \Big)^{1/q} \Big\| \Big(\sum_{i=1}^n g_i^2\Big)^{1/2}\Big\|_{L_q(\ms{C}_n;\R)},
\end{equation}
where $\{\ms{E}_i\}_{i=0}^n$ are the conditional expectations \eqref{eq:defEi}. Using the fact that $L_\infty(\ms{C}_n;\R)$ is isomorphic to $L_{n}(\ms{C}_n;\R)$, we thus deduce that
\begin{equation}
\begin{split}
\Big\| \Big(\sum_{i=1}^n \big(\ms{E}_i g_i\big)^2\Big)^{1/2}\Big\|_{L_\infty(\ms{C}_n;\R)} & \gtrsim \Big( \int_0^n y^{n/2} e^{-y}\diff y \Big)^{1/n} \Big\| \Big(\sum_{i=1}^n g_i^2\Big)^{1/2}\Big\|_{L_\infty(\ms{C}_n;\R)} 
\\ & \asymp \sqrt{n}\Big\| \Big(\sum_{i=1}^n g_i^2\Big)^{1/2}\Big\|_{L_\infty(\ms{C}_n;\R)}
\end{split}.
\end{equation}
Therefore, by duality in $L_\infty(\ms{C}_n;\ell_2^n)$ and Khintchine's inequality \cite{Khi23}, we deduce that there exists $n$ functions $h_1,\ldots,h_n:\ms{C}_n\to\R$ such that
\begin{equation} \label{eq:hidef}
\frac{1}{2^n}\sum_{\delta\in\ms{C}_n}\Big\| \sum_{i=1}^n \delta_i \ms{E}_i h_i\Big\|_{L_1(\ms{C}_n;\R)} \gtrsim \frac{\sqrt{n}}{2^n} \sum_{\delta\in\ms{C}_n} \Big\| \sum_{i=1}^n \delta_i  h_i\Big\|_{L_1(\ms{C}_n;\R)}.
\end{equation}
Suppose that a Banach space $X$ with $\mf{s}_p(X)<\infty$ is not $K$-convex, so that there exists a constant $K\in[1,\infty)$ such that for every $n\in\N$, there exists a linear operator $\msf{J}_n:L_1(\ms{C}_n;\R)\to X$ satisfying
 \begin{equation}
 \forall \ h\in L_1(\ms{C}_n;\R), \ \ \ \ \|h\|_{L_1(\ms{C}_n;\R)} \leq \|\msf{J}_nh\|_X \leq K\|h\|_{L_1(\ms{C}_n;\R)}.
 \end{equation}
 Consider the functions $H_1,\ldots,H_n:\ms{C}_n\to L_1(\ms{C}_n;\R)$ given by
 \begin{equation}
 \forall \ \e,\e'\in\ms{C}_n, \ \ \ \big[H_i(\e)\big](\e') = h_i(\e_1\e_1',\ldots,\e_n\e_n'),
 \end{equation}
where $h_i\in L_1(\ms{C}_n;\R)$ are the functions satisfying \eqref{eq:hidef}. Then, for every $i\in\{1,\ldots,n\}$, we have $[\ms{E}_iH_i(\e)](\e') = \ms{E}_ih_i(\e_1\e_1',\ldots,\e_n\e_n')$ and, by translation invariance, for every $\e,\delta\in\ms{C}_n$ we have
 \begin{equation*}
\Big\| \sum_{i=1}^n \delta_i \ms{E}_i H_i(\e)\Big\|_{L_1(\ms{C}_n;\R)} = \Big\| \sum_{i=1}^n \delta_i \ms{E}_i h_i\Big\|_{L_1(\ms{C}_n;\R)} \ \ \mbox{and} \ \ \Big\| \sum_{i=1}^n \delta_i H_i(\e)\Big\|_{L_1(\ms{C}_n;\R)} = \Big\| \sum_{i=1}^n \delta_i h_i\Big\|_{L_1(\ms{C}_n;\R)}
 \end{equation*}
 Therefore, considering the mappings $f_1,\ldots,f_n:\ms{C}_n\to X$ given by $f_i = \msf{J}_n \circ H_i$, we see that
 \begin{equation}
\Big(\frac{1}{2^n} \sum_{\delta\in\ms{C}_n} \Big\| \sum_{i=1}^n \delta_i\ms{E}_i f_i\Big\|_{L_p(\ms{C}_n;X)}^p\Big)^{1/p} \gtrsim K^{-1} \sqrt{n} \Big( \frac{1}{2^n} \sum_{\delta\in\ms{C}_n} \Big\| \sum_{i=1}^n \delta_if_i\Big\|_{L_p(\ms{C}_n;X)}^p\Big)^{1/p},
 \end{equation}
 thus showing that $\mf{s}_p(X)\gtrsim K^{-1} \sqrt{n}$, which is a contradiction.
 \end{proof}

Recall that the $X$-valued Rademacher projection is defined to be
\begin{equation}
\msf{Rad}\Big(  \sum_{A\subseteq\{1,\ldots,n\}} \widehat{f}(A) w_A\Big) \eqdef \sum_{i=1}^n \widehat{f}(\{i\}) w_{\{i\}}.
\end{equation}
A deep theorem of Pisier \cite{Pis82} asserts that a Banach space is $K$-convex if and only if
\begin{equation}
\forall \ r\in(1,\infty), \ \ \ \ \msf{K}_r(X) \eqdef \sup_{n\in\N}\big\|\msf{Rad}\big\|_{L_r(\ms{C}_n;X)\to L_r(\ms{C}_n;X)} <\infty.
\end{equation}
In particular, it follows from Lemma \ref{f2} that $\mf{s}_p(X)<\infty$ for some $p\in(1,\infty)$ implies that $\msf{K}_r(X)<\infty$ for every $r\in(1,\infty)$. We proceed by showing that Banach spaces belonging to the class considered in \cite[Theorem~1.4]{HN13} satisfy the assumptions of Theorem \ref{thm}.

\begin{proposition} \label{fact:hn1}
Let $(X,\|\cdot\|_X)$ be a Banach space. If $X^\ast$ is a UMD$^+$ space, then $X$ is a UMD$^-$ space and $\mf{s}_p(X)<\infty$ for every $p\in(1,\infty)$.
\end{proposition}

\begin{proof}
The fact that if $X^\ast$ is UMD$^+$, then $X$ is UMD$^-$ has been proven by Garling in \cite[Theorem~1]{Gar90}, so we only have to prove that $\mf{s}_p(X)<\infty$. Let $f_1,\ldots,f_n:\ms{C}_n\to X$ and $G^\ast:\ms{C}_n\times\ms{C}_n\to X^\ast$ be such that
\begin{equation} \label{is normal}
\Big(\frac{1}{2^n} \sum_{\delta\in\ms{C}_n} \Big\| \sum_{i=1}^n \delta_i \ms{E}_i f_i\Big\|_{L_p(\ms{C}_n;X)}^p\Big)^{1/p} = \frac{1}{4^n} \sum_{\e,\delta\in\ms{C}_n} \big\langle G^\ast(\e,\delta),\sum_{i=1}^n\delta_i \ms{E}_if_i(\e)\big\rangle
\end{equation}
and $\|G^\ast\|_{L_q(\ms{C}_n\times\ms{C}_n;X^\ast)} = 1$, where $\tfrac{1}{p}+\tfrac{1}{q}=1$. Let $G_i^\ast:\ms{C}_n\to X^\ast$ be given by
\begin{equation} \label{it is rad}
\forall \ \e\in\ms{C}_n, \ \ \ \ G_i^\ast(\e) = \frac{1}{2^n}\sum_{\delta\in\ms{C}_n} \delta_i G_i^\ast(\e,\delta).
\end{equation}
Then, since $X^\ast$ is UMD$^+$, we deduce that $X^\ast$ is also $K$-convex (this is proven in \cite{Gar90} but it also follows by combining Bourgain's inequality \eqref{eq:bourgain} with Lemma \ref{f2}) and thus
\begin{equation} \label{eq:is k con}
\Big(\frac{1}{2^n} \sum_{\delta\in\ms{C}_n} \Big\| \sum_{i=1}^n \delta_i  G^\ast_i\Big\|_{L_q(\ms{C}_n;X^\ast)}^q\Big)^{1/q} \stackrel{\eqref{it is rad}}{=} \Big(\frac{1}{4^n}\sum_{\e,\delta\in\ms{C}_n} \big\| \msf{Rad}_\delta G^\ast(\e,\delta)\big\|_X^q\Big)^{1/q} \leq \msf{K}_q(X^\ast).
\end{equation}
Hence, we have
\begin{equation} \label{eq:identities}
\begin{split}
&\Big(\frac{1}{2^n} \sum_{\delta\in\ms{C}_n}  \Big\| \sum_{i=1}^n \delta_i \ms{E}_i f_i\Big\|_{L_p(\ms{C}_n;X)}^p\Big)^{1/p} \stackrel{\eqref{is normal}\wedge\eqref{it is rad}}{=} \frac{1}{4^n} \sum_{\e,\delta\in\ms{C}_n} \big\langle \sum_{i=1}^n\delta_i G_i^\ast(\e),\sum_{i=1}^n\delta_i \ms{E}_if_i(\e)\big\rangle 
\\ & = \frac{1}{2^n}\sum_{\e\in\ms{C}_n} \langle G_i^\ast(\e), \ms{E}_if_i(\e)\rangle = \frac{1}{2^n}\sum_{\e\in\ms{C}_n} \langle \ms{E}_iG_i^\ast(\e), f_i(\e)\rangle = \frac{1}{4^n} \sum_{\e,\delta\in\ms{C}_n} \big\langle \sum_{i=1}^n\delta_i \ms{E}_iG_i^\ast(\e),\sum_{i=1}^n\delta_i f_i(\e)\big\rangle 
\\ & \leq \Big(\frac{1}{2^n} \sum_{\delta\in\ms{C}_n} \Big\| \sum_{i=1}^n \delta_i \ms{E}_i G^\ast_i\Big\|_{L_q(\ms{C}_n;X^\ast)}^q\Big)^{1/q} \cdot \Big(\frac{1}{2^n} \sum_{\delta\in\ms{C}_n} \Big\| \sum_{i=1}^n \delta_i f_i\Big\|_{L_p(\ms{C}_n;X)}^p\Big)^{1/p}.
\end{split}
\end{equation}
Therefore, combining \eqref{eq:identities} with \eqref{eq:bourgain} and \eqref{eq:is k con}, we deduce that
\begin{equation}
\begin{split}
\Big(\frac{1}{2^n} \sum_{\delta\in\ms{C}_n} & \Big\| \sum_{i=1}^n \delta_i \ms{E}_i f_i\Big\|_{L_p(\ms{C}_n;X)}^p\Big)^{1/p}
\\ & \stackrel{\eqref{eq:bourgain}}{\leq} \mf{s}_q(X^\ast)  \Big(\frac{1}{2^n} \sum_{\delta\in\ms{C}_n} \Big\| \sum_{i=1}^n \delta_i  G^\ast_i\Big\|_{L_q(\ms{C}_n;X^\ast)}^q\Big)^{1/q} \cdot \Big(\frac{1}{2^n} \sum_{\delta\in\ms{C}_n} \Big\| \sum_{i=1}^n \delta_i f_i\Big\|_{L_p(\ms{C}_n;X)}^p\Big)^{1/p}
\\ & \stackrel{\eqref{eq:is k con}}{\leq}  \mf{s}_q(X^\ast)  \msf{K}_q(X^\ast) \cdot \Big(\frac{1}{2^n} \sum_{\delta\in\ms{C}_n} \Big\| \sum_{i=1}^n \delta_i f_i\Big\|_{L_p(\ms{C}_n;X)}^p\Big)^{1/p},
\end{split}
\end{equation}
which shows that $\mf{s}_p(X) \leq \msf{K}_q(X^\ast) \mf{s}_q(X^\ast)$.
\end{proof}

We conclude by observing that spaces satisfying the assumptions of Theorem \ref{thm} are necessarily superreflexive (see \cite{Pis16} for the relevant terminology).

\begin{lemma}
If a UMD$^-$ Banach space $(X,\|\cdot\|_X)$ satisfies $\mf{s}_p(X)<\infty$, then $X$ is superreflexive.
\end{lemma}

\begin{proof}
A theorem of Pisier \cite{Pis73} asserts that a Banach space $X$ is $K$-convex if and only if $X$ has nontrivial Rademacher type. Therefore, we deduce from Lemma \ref{f2} that if $\mf{s}_p(X)<\infty$ for some $p\in(1,\infty)$, then there exists $s\in(1,2]$ and $T_s(X)\in(0,\infty)$ such that
\begin{equation} \label{eq:radtype}
\forall \ x_1,\ldots,x_n\in X, \ \ \ \ \Big(\frac{1}{2^n} \sum_{\delta\in\ms{C}_n}\Big\|\sum_{i=1}^n \delta_i x_i\Big\|_X^s\Big)^{1/s} \leq T_s(X) \Big(\sum_{i=1}^n \|x_i\|_X^s\Big)^{1/s}.
\end{equation}
Therefore, if $X$ also satisfies the UMD$^-$ property, we deduce that for every $X$-valued martingale $\{\ms{M}_i:\Omega\to X\}_{i=0}^n$,
\begin{equation}
\begin{split}
\|\ms{M}_n - \ms{M}_0\|_{L_s(\Omega,\mu;X)} & \leq \beta_s^-(X) \Big(\frac{1}{2^n} \sum_{\delta\in\ms{C}_n} \Big\| \sum_{i=1}^n \delta_i (\ms{M}_i-\ms{M}_{i-1}) \Big\|^s_{L_s(\Omega,\mu;X)}\Big)^{1/s}
\\ & \stackrel{\eqref{eq:radtype}}{\leq} \beta_s^-(X)T_s(X) \Big( \sum_{i=1}^n \|\ms{M}_i-\ms{M}_{i-1}\|^s_{L_s(\Omega,\mu;X)}\Big)^{1/s},
\end{split}
\end{equation}
which means that $X$ has martingale type $s$. Combining this with well known results linking martingale type and superreflexivity (see \cite{Pis16}), we reach the desired conclusion.
\end{proof}

Therefore, Theorem \ref{thm} establishes that $\mf{P}_p^n(X)=\Theta(1)$ for $X$ in a (strict, see \cite{Gar90, Qiu12}) subclass of all superreflexive spaces. In the forthcoming manuscript \cite{EN20}, the bound $\mf{P}_p^n(X) = o(\log n)$ is shown to hold for every superreflexive Banach space $X$ and $p\in(1,\infty)$.

\bibliography{pisier_umd}
\bibliographystyle{alpha}
%\nocite{*}

\end{document}